\documentclass[10pt]{article}

\usepackage{enumerate}

\usepackage[english]{babel}
\usepackage{amsmath,amssymb,amsthm}

\newtheorem{theorem}{Theorem}[section]
\newtheorem{lemma}[theorem]{Lemma}
\newtheorem{proposition}[theorem]{Proposition}

\newtheorem{stw}[theorem]{Proposition}

\newcommand{\supp}{\operatorname{supp}}

\newcommand{\GK}{\operatorname{GKdim}}

\newcommand{\clK}{\operatorname{clKdim}}

\newcommand{\s}{\subseteq}

\DeclareMathOperator{\HK}{HK}

\begin{document}

\title{On the radical of a Hecke--Kiselman algebra}

\author{Jan Okni\'nski and Magdalena Wiertel}
\date{}
\maketitle

\begin{abstract}
The Hecke-Kiselman algebra of a finite oriented graph $\Theta$
over a field $K$ is studied. If $\Theta$ is an oriented cycle, it
is shown that the algebra is semiprime and its central
localization is a finite direct product of matrix algebras over
the field of rational functions $K(x)$. More generally, the
radical is described in the case of PI-algebras, and it is shown
that it comes from an explicitly described congruence on the
underlying Hecke-Kiselman monoid. Moreover, the algebra modulo the
radical is again a Hecke-Kiselman algebra and it is a finite
module over its center.
\end{abstract}

\vspace{20pt}



\vspace{20pt}

\section{Introduction}
Let $\Theta$ be a finite  simple oriented graph with $n$ vertices
$\{1, \ldots, n\}$. The Hecke--Kiselman monoid $\HK_{\Theta}$
associated to $\Theta$, \cite{maz}, is generated by elements $x_1,
\ldots, x_n$ subject to the defining relations:
\begin{itemize}
    \item[(i)] $ x_i=x_{i}^2$,
    for $1 \leqslant i \leqslant n$,
    \item[(ii)] $x_ix_j = x_jx_i$ if the vertices $i$, $j$ are not connected in  $\Theta$,
    \item[(iii)] $x_ix_jx_i = x_jx_ix_j = x_ix_j$, if $i$, $j$ are connected by an arrow
    $i \to j$ in $\Theta$.
\end{itemize}
Thus, $\HK_{\Theta}$ is a natural homomorphic image of the
corresponding Coxeter monoid, where relations (iii) are replaced
by the braid relations $x_ix_jx_i = x_jx_ix_j$. Several
combinatorial properties of $\HK_{\Theta}$ and their
representations were studied in \cite{ara},\cite{type
A},\cite{maz},\cite{OnK},\cite{mecel_okninski2}. We continue the
study in \cite{wiertel-JO}, where the structure of $\HK_{\Theta}$,
and of the associated algebra $K[\HK_{\Theta}]$ over a field $K$,
is investigated. The case where $\Theta$ is the oriented cycle
$x_{1}\rightarrow x_{2}\rightarrow \cdots \rightarrow x_{n}
\rightarrow x_{1}$, with $n\geqslant 3$, plays a crucial role. Our
first main result shows that the associated Hecke-Kiselman
algebra, denoted by $K[C_n]$, is semiprime. It is also Noetherian,
as shown in \cite{wiertel-JO}. Consequently, since $K[C_n]$ is an
algebra of Gelfand-Kirillov dimension one \cite{mecel_okninski1},
from \cite{stafford} it follows that $K[C_n]$ is a finite module
over its center. Moreover, its classical quotient ring can be
completely described.

\begin{theorem}\label{main} Let $n\geqslant 3$. Then $K[C_n]$ is a semiprime Noetherian
PI-algebra. Moreover, its classical quotient ring is isomorphic to
$\prod_{i=0}^{n-2} M_{n_i}(K(x))$, where $n_i =\binom{n}{i+1}$,
for $i=0,\ldots, n-2$.
\end{theorem}

Recall that the classical quotient ring of a semiprime Goldie
PI-algebra is its central localization, see \cite{rowenPI},
Theorem~1.7.34.

In particular, this result answers a question asked in
\cite{wiertel-JO}. Next, we apply it to derive a description of
the Jacobson radical ${\mathcal J}(K[\HK_{\Theta}])$ of an
arbitrary algebra $K[\HK_{\Theta}]$, provided it satisfies a
polynomial identity. The latter condition is equivalent to a
simple condition expressed in terms of the graph $\Theta$,
\cite{mecel_okninski1}. Namely, it is equivalent to saying that
$\Theta$ does not contain two cyclic subgraphs (i.e. subgraphs
which are oriented cycles) connected by an oriented path. We prove
that the radical is the ideal determined by an explicitly
described congruence $\rho$ on $\HK_{\Theta}$, so that
$K[\HK_{\Theta}]/ {\mathcal J}(K[\HK_{\Theta}]) \cong
K[\HK_{\theta}/\rho]$ is again a Hecke-Kiselman algebra and it has
a very transparent structure. For a congruence $\eta$ on a
semigroup $S$, the kernel of the natural homomorphism
$K[S]\longrightarrow K[S/ \eta]$ will be denoted by $I(\eta)$. So
$K[S/ \eta ] \cong K[S]/I(\eta)$.

Namely, let $\rho$ be the congruence on $\HK_{\Theta}$ generated
by all pairs $(xy,yx)$ such that there is an arrow $x\rightarrow
y$ that is not contained in any cyclic subgraph of $\Theta$. Let
$\Theta'$ be the subgraph of $\Theta$ obtained by deleting all
arrows $x\rightarrow y$ that are not contained in any cyclic
subgraph of $\Theta$. Then $\HK_{\Theta'}\cong \HK_{\Theta}/\rho$.
(If there is no such a pair then we assume that $\rho$ is the
trivial congruence.) Then, because of the assumption that
$K[\HK_{\Theta}]$ is a PI-algebra, the connected components of
$\Theta'$ are either singletons or cyclic subgraphs. Recall from
\cite{wiertel-JO} that this implies that $K[\HK_{\Theta'}]$ is a
Noetherian algebra. Indeed, Noetherian algebras $K[\HK_{\Theta}]$
are characterized by the condition: each of the connected
components of the graph $\Theta$ either is acyclic or it is a
cyclic graph of length $n$ for some $n\geqslant 3$. Our second
main result reads as follows.

\begin{theorem} \label{second_main}
Assume that $\Theta $ is a finite oriented graph such that
$K[\HK_{\Theta}]$ is a PI-algebra. Let $\Theta'$ be the subgraph
of $\Theta$
    obtained by deleting all arrows $x\rightarrow y$ that are not
    contained in any cyclic subgraph of $\Theta$ and let $\rho$ be the
    congruence on $\HK_{\Theta}$ defined above. Then
\begin{enumerate}
\item the Jacobson radical ${\mathcal
J}(K[\HK_{\Theta}])$ of $K[\HK_{\Theta}]$ is equal to the ideal
$I(\rho)$ determined by $\rho$,
\item  $K[\HK_{\Theta}]/{\mathcal
J}(K[\HK_{\Theta}])\cong K[\HK_{\Theta'}]$ and it is the tensor
product of algebras $K[\HK_{\Theta_{i}}]$ of the connected
components $\Theta_1,\ldots, \Theta_m$ of $\Theta'$, each being
isomorphic to $K\oplus K$ or to the algebra $K[C_{j}]$, for some
$j\geqslant 3$,
\item $K[\HK_{\Theta'}]$ is a finitely generated module
over its center.
\end{enumerate}
\end{theorem}

Recall that the Jacobson radical of a finitely generated
PI-algebra $R$ is nilpotent, see \cite{rowen}, Theorem 6.3.39.
However, for $R=K[\HK_{\Theta}]$ this can also be derived from our
proof.

\section{Some background} \label{background}

 A Gr\"obner basis for $C_n$ has been found in
\cite{mecel_okninski2}, by applying the diamond lemma, see
\cite{diamond}. Consequently, the elements of $C_n$ can be treated
as words in the free monoid $F=\langle x_{1},\ldots , x_n\rangle$
that are reduced in terms of certain rewriting system in $F$. Let
$|w|_{i}$ denote the degree of a word $w$ (treated as an element
of $C_n$) in the generator $x_i$. If $i,j\in \{1,\ldots, n\}$ then
$x_i \cdots x_j$ denotes the product of all consecutive generators
from $x_i$ up to $x_j$ if $ i<j$, or down to $x_{j}$, if $i > j$.
\begin{theorem}\label{basisCn}
Let $\Theta = C_n$ for some $n\geqslant 3$. Let $S$ be the system
of reductions in $F$ consisting of all pairs of the form
\begin{enumerate}
    \item[(1)] $(x_ix_i,x_i)$ for all $i\in\{1,\dotsc,n\}$,
    \item[(2)] $(x_jx_i,x_ix_j)$ for all $i,j\in\{1,\dotsc,n\}$ such that $1<j-i<n-1$,
    \item[(3)] $(x_n(x_1\dotsm x_i)x_j,x_jx_n(x_1\dotsm x_i))$ for all
    $i,j\in\{1,\dotsc,n\}$ such that $i+1<j<n-1$,
    \item[(4)] $(x_iux_i,x_iu)$ for all $i\in\{1,\dotsc,n\}$ and $1\ne u\in F$
    such that $|u|_i=|u|_{i-1}=0$. Here, we
    write  $i-1 = n$ if $i = 1$,
    \item[(5)] $(x_ivx_i,vx_i)$ for all $i\in\{1,\dotsc,n\}$ and $1\ne v\in F$
    such that $|v|_i=|v|_{i+1}=0$. Here we write
     $i+1 = 1$ if $i =n$.
\end{enumerate}
Then the set $\{w - v \mid \text{ for } (w, v) \in S\}$ is a
Gr\"obner basis of the algebra $K[C_n]$.
\end{theorem}

It follows that an element $w\in F$ is a reduced word if and only
if $w$ has no factors that are leading terms of the reductions (1)
- (5) listed above. This reduction system is compatible with the
degree-lexicographical ordering on the free monoid $F$ defined by
$x_1<x_2 <\cdots <x_n$. We will use this result from
\cite{mecel_okninski2} several times without further comment.

 Our approach heavily depends on the results of
\cite{wiertel-JO}. In particular, a very transparent description
of the reduced forms of almost all elements of $C_n$ has been
found in \cite{wiertel-JO}, Theorem 2.1. Namely, for
$i=0,1,\ldots, n-2$, the set $\tilde{M_i}$ of reduced forms of
elements of $C_n$ that have a factor of the form
$x_nq_i=x_nx_1\cdots x_ix_{n-1}\cdots x_{i+1}$ can be described as
follows
\begin{equation}
    \tilde{M_i}=\{a(x_nq_i)^kb\in C_n: a\in A_i, b\in B_i, k\geqslant
    1\},
 \end{equation}
where $A_{i},B_{i}$ are certain well defined sets. Moreover, if
$\tilde{M}=\bigcup_{i=0}^{n-2}\tilde{M}_i$ then the set
$C_n\setminus  \tilde{M}$ is finite and each
$M_i=\tilde{M_{i}^{0}}$ ($\tilde{M_{i}}$ with zero adjoined) is
isomorphic to a semigroup of matrix type $\mathcal{M}^{0}(S_i,
A_i, B_i; P_i)$, where $S_i$ denotes the cyclic semigroup
generated by $s_i=x_nq_i$, $P_i$ is a matrix of size $B_i\times
A_i$ with coefficients in $\langle x_nq_i\rangle\cup\{\theta\}$,
where $\langle s_i\rangle$ is the monoid generated by $s_i$.
Recall that, if $S$ is a semigroup, $A,B$ are nonempty sets and
$P=(p_{ba})$ is a $B\times A$ - matrix with entries in $S^{0}$,
then the semigroup of matrix type ${\mathcal M}^{0}(S,A,B;P)$ over
$S$ is the set of all triples $(s,a,b)$, where $s\in S, a\in A,
b\in B$, with the zero element $\theta$, with operation
$(s,a,b)(s',a',b')= (sp_{ba'}s',a,b')$ if $p_{ba'}\in S$ and
$\theta $ otherwise. So, ${\mathcal M}^{0}(S,A,B;P)$ is an order
in the completely $0$-simple semigroup ${\mathcal M}^{0}(G,A,B;P)$
over a cyclic infinite group, in the sense of \cite{fountain}.
Moreover ${\mathcal M}^{0}(K[S],A,B;P)$ denotes the corresponding
algebra of matrix type. It is defined as the contracted semigroup
algebra $K_{0}[{\mathcal M}^{0}(S,A,B;P)]$ and (if $A,B$ are
finite) it can be interpreted as the set of all $A\times B$ -
matrices over $K[S]$ with operation $\alpha \beta = \alpha \circ
P\circ \beta $, where $\circ$ stands for the standard matrix
product. For basic results on semigroups and algebras of matrix
type we refer to \cite{semalg}, Chapter 5. They play a fundamental
role in representation theory of semigroup algebras.

It is shown in \cite{wiertel-JO} that $|A_i|=|B_i|$ and $P_i$ is
not a zero divisor in the matrix ring $M_{n_{i}}(K[s_i])$.
Therefore, $P_i$ is invertible as a matrix in $M_{n_{i}}(K(s_i))$
and hence the algebra of matrix type $\mathcal{M}^{0}(K(s_i), A_i,
B_i; P_i)\cong M_{n_{i}}(K(s_i))$; this isomorphism is
accomplished via the map $x\mapsto x\circ P$. Moreover, the latter
is a central localization of the prime algebra $K_{0}[M_i]\cong
\mathcal{M}^{0}(K[S_i], A_i, B_i; P_i)$, where $S_i$ is the cyclic
semigroup generated by $s_i$.

\begin{lemma} \label{matrix type}
$K_0[M_i]$ is a prime algebra. Moreover, it does not have nonzero
finite dimensional ideals.
\end{lemma}
\begin{proof}
The first assertion was proved in \cite{wiertel-JO}, Theorem 5.8.
Suppose that $J$ is a nonzero ideal. Then there exist $v,w\in M_i$
such that $vJw\neq 0$. Hence, the matrix type structure of
$K_0[M_i]$ implies easily that there exist $v',w'\in M_i$ such
that $0\neq v'v Jww' \subseteq K[x_nq_i]$. Then, clearly, $J\cap
K[x_nq_i]$ is infinite dimensional.
\end{proof}

 We start with calculating the size of the set $A_i$, for
every $i=0,\ldots, n-2$ and $n\geqslant3$.

\begin{stw}\label{rozmiar}
    For any $i\in\{0,\ldots, n-2\}$ and $n\geqslant 3$ we have $|A_i|=\binom{n}{i+1}$.
\end{stw}
\begin{proof}
    For $i=n-2$ the assertion follows from Lemma 2.5 in \cite{wiertel-JO},
    so next we assume that $i\leqslant n-3$.\\
    From the description of the set $A_i$ from Theorem 2.1 in \cite{wiertel-JO}
    it is clear that every element $w$ of $A_i$ is exactly of one of the forms
    \begin{enumerate}
        \item $w=(x_{k_s}\cdots x_s)(x_{k_{s+1}}\cdots x_{s+1})\cdots(x_{k_{i+1}}\cdots x_{i+1})$ where $i+1\geqslant s\geqslant 1$, $s+1<k_{s+1}<\cdots <k_{i+1}\leqslant n-1$ and $s\geqslant k_s$; for $s=i+1$ we assume that $w=(x_{k_{i+1}}\cdots x_{i+1})$ with $i+1\geqslant k_{i+1}$;
        \item $w=(x_{k_s}\cdots x_s)(x_{k_{s+1}}\cdots x_{s+1})\cdots(x_{k_{i+1}}\cdots x_{i+1})$ where $i+1\geqslant s\geqslant 1$, $s<k_{s}<\cdots <k_{i+1}\leqslant n-1$;
        \item $w=1$.
    \end{enumerate}
    Choose $1\leqslant s\leqslant i+1$ and $0\leqslant i\leqslant n-3$.
    Then the elements $w$ from Case 1. are in a bijection with strictly
    increasing sequences $(k_s,\ldots, k_{i+1})$ of natural numbers such
    that $1\leqslant k_s\leqslant s<s+2\leqslant k_{s+1}<\cdots<k_{i+1}\leqslant n-1$.
    It is easy to see that there exist exactly $s\binom{n-s-2}{i-s+1}$ sequences
    of the above form.
    Similarly, elements $w$ of the form as in Case 2. are in a bijection with
    strictly increasing sequences $(k_s, \ldots, k_{i+1})$ of natural numbers such
    that $s+1\leqslant k_s<\cdots <k_{i+1}\leqslant n-1$. There are exactly
    $\binom{n-s-1}{i-s+2}$ such sequences.

    It follows that
    $$|A_i|=1+\sum_{s=1}^{i+1}\left(\binom{n-s-1}{i-s+2}+s\binom{n-s-2}{i-s+1}\right).$$

    Thus, it is enough to prove that $1+\sum_{s=1}^{i+1} (\binom{n-s-1}{i-s+2}+s\binom{n-s-2}{i-s+1})=\binom{n}{i+1}$
    for $n\geqslant 3$ and $0\leqslant i\leqslant n-3$.

   Moreover, if $i=n-3$, then by a direct calculation we get that
    $$ 1+\sum_{s=1}^{n-2}\left(\binom{n-s-1}{n-s-1}+s\binom{n-s-2}{n-s-2}\right)=\binom{n}{n-2},$$
    as desired.

   It is easy to check that
    $$1+\sum_{s=1}^{i+1}\left( \binom{n-s-1}{i-s+2}+s\binom{n-s-2}{i-s+1}\right)=\sum_{k=0}^{i+1}(i+2-k)\binom{n-i-3+k}{k}.$$
     Indeed, substituting $k=i+1-s$ in the sum in the left hand side, we get that
    this sum is equal to
    \begin{align*}
    1+\sum_{k=0}^{i}\binom{n-i-2+k}{k+1} + \sum_{k=0}^{i}(i+1-k)\binom{n-i-3+k}{k}=\\=
    1+\sum_{k=1}^{i+1}\binom{n-i-3+k}{k} + \sum_{k=0}^{i}(i+1-k)\binom{n-i-3+k}{k}=\\=
    \sum_{k=0}^{i+1}\binom{n-i-3+k}{k} + \sum_{k=0}^{i+1}(i+1-k)\binom{n-i-3+k}{k}=\\=
    \sum_{k=0}^{i+1}(i+2-k)\binom{n-i-3+k}{k},
    \end{align*}
    as claimed.

    We proceed by induction on $n$ to prove that
    $$ \sum_{k=0}^{i+1}(i+2-k)\binom{n-i-3+k}{k}=\binom{n}{i+1}.$$
    For $i=0$ and arbitrary $n\geqslant3$ we have $1+\binom{n-2}{1}+\binom{n-3}{0}=\binom{n}{1}$
    and the assertion follows.
    If $n=3$, then we have $0\leqslant i\leqslant 0$, so the proposition holds.

    Assume now that the equality is true for some $n$ and every $i\leqslant n-3$.
    Consider the sum
    $$\sum_{k=0}^{i+1}(i+2-k)\binom{(n+1)-i-3+k}{k}$$
    for $n-2> i>0$.
    Using $\binom{m+1}{k}=\binom{m}{k}+\binom{m}{k-1}$ if $k\geqslant 1$ and
    $\binom{m+1}{0}=\binom{m}{0}$ we get
    \begin{eqnarray*} \lefteqn{\sum_{k=0}^{i+1}(i+2-k)\binom{(n+1)-i-3+k}{k}=}\\
    &  = & \sum_{k=0}^{i+1}(i+2-k)\binom{n-i-3+k}{k}+\sum_{k=1}^{i+1}(i+2-k)\binom{n-i-3+k}{k-1}.
    \end{eqnarray*}
    From the induction hypothesis it follows that the first sum is equal to $\binom{n}{i+1}$.
    Substituting $m=k-1$ and $j=i-1$ we get
    $$\sum_{k=1}^{i+1}(i+2-k)\binom{n-i-3+k}{k-1}=\sum_{m=0}^{j+1}(j+2-m)\binom{n-j-3+m}{m}.$$
    From the induction hypothesis it follows that the above sum is equal to $\binom{n}{i}$.
    Now, using $\binom{n}{i+1}+\binom{n}{i}=\binom{n+1}{i+1}$ we get
    $$\sum_{k=0}^{i+1}(i+2-k)\binom{(n+1)-i-3+k}{k}=\binom{n}{i+1}+\binom{n}{i}=\binom{n+1}{i+1}$$
    and the assertion follows.
\end{proof}

\section{Main results}

We will identify, without further comments, elements of the monoid
$C_n$ with words in free monoid $F$ that are reduced with respect
to the system $S$ described in Theorem~\ref{basisCn}.

Our first main aim is to show that $K[C_n]$ is semiprime. To prove
this, we strengthen some of the results from \cite{wiertel-JO}.

Consider the automorphism $\sigma$ of $C_n$ given by
$\sigma(x_i)=x_{i+1}$ for $i=1,\ldots, n$, where we agree that
$x_{n+1}=x_1$. The natural extension to an automorphism of
$K[C_n]$ also will be denoted by $\sigma$. For basic properties of
this automorphism we refer to Section 3 in \cite{wiertel-JO}.

We have an ideal chain in $C_n$ \begin{equation}\label{chain}
\emptyset= I_{n-2}\subseteq I_{n-3}\subseteq I_{n-4} \subseteq
\cdots \subseteq I_{0}\subseteq I_{-1} \end{equation} where
$I_i=\{w\in C_n: C_nwC_n\cap\langle x_nq_i\rangle=\emptyset\}$ for
$i=0,\ldots, n-2$, and
$$I_{-1}=I_0\cup C_n x_nq_0C_n.$$
In particular, using Corollary 3.17 in \cite{wiertel-JO} we obtain
that $\sigma(I_k)=I_k$ for $k=0,\ldots, n-3$.  The key structural
result obtained in \cite{wiertel-JO} reads as follows.

\begin{proposition}\label{struktura}
    Consider the ideal chain (\ref{chain}). Then
    \begin{enumerate}
        \item for $i=0,\ldots, n-2$, the semigroups of matrix type
        $M_i=\mathcal{M}^0(S_i, A_i, B_i; P_i)$,
        satisfy $M_i\s I_{i-1}/I_i$,
        \item for $i=1,\ldots, n-2$, the sets $(I_{i-1}/I_i)\setminus M_i$ are finite;
        \item $I_{-1}/I_0=M_0$;
        \item $\tilde{M}_{n-2}=M_{n-2}\setminus \{\theta\}$ is an ideal in $C_n$;
        \item  $C_n/I_{-1}$ is a finite semigroup.
    \end{enumerate}
\end{proposition}
 The following observation can be deduced from the results and
methods of \cite{wiertel-JO}.
\begin{lemma}\label{ideal}
$M_i$ is a right ideal in $C_n/I_i$ for every $i=0,1,\ldots, n-2$.
\end{lemma}
\begin{proof}
 Let $a(x_nq_i)^kb\in \tilde{M_i}$ and take any generator $x_r\in
C_n$. Assume that $a(x_nq_i)^kbx_r\notin\tilde{M_i}$. We claim
that then $a(x_nq_i)^kbx_r\in I_i$. Let $b'$ be the reduced form
of $bx_r$. If $b'=x_j\bar{b}$ for some word $\bar{b}$, where
$j\leqslant i+1$, then using reduction $(4)$ from Theorem
\ref{basisCn} we get that $a(x_nq_i)^kb'$ can be reduced to
$a(x_nq_i)^k\bar{b}$. Therefore we can assume that a prefix of
$b'$ is equal to $x_j$, for some $j>i+1$. If $j<n$, then it can be
calculated that $a(x_nq_i)^kbx_r$ can be rewritten as a word with
a factor of the form $x_{j-1}\cdots x_{i+2}x_nx_1\cdots
x_{i+1}x_{n-1}\cdots x_{j}$ and this element is in $I_i$ by Lemma
3.8 in \cite{wiertel-JO}. Let us now consider the case when $x_n$
is a prefix of $b'$. As we assume that $a(x_nq_i)^kb'\notin
\tilde{M_i}$, this word can be rewritten in $C_n$ as an element
without the factor $x_nq_i$. From Theorem \ref{basisCn} it is easy
to see that to obtain a word without such a factor one has to use
a reduction of type (5). Therefore $a(x_nq_i)^kb'$ can be written
as a word with a prefix of the form $a(x_nq_i)^kx_nvx_j$,
where $|x_nv|_j=|x_nv|_{j+1}=0$. Moreover, for $j\leqslant i$ or
$j=n-1$ the generator $x_{j+1}$ occurs in $x_nq_ix_n$ after $x_j$,
thus the reduction of $x_j$ of type (5) is not possible in this
case. Therefore $n-1>j\geqslant i+1$. It follows (see Lemma 2.3 in
\cite{wiertel-JO}) that such a prefix is of the form
$a(x_nq_i)^kx_nx_1\cdots x_{j}$. Therefore this element has a
factor $x_nx_1\cdots x_{i}x_{n-1}\cdots x_{i+1}x_nx_1\cdots x_j$
for some $n-1>j\geqslant i+1$. It can be checked (using the
reductions from Theorem \ref{basisCn}) that the latter word can be
rewritten as an element with the factor $x_{n-1}\cdots
x_{j+1}x_{n}x_{1}\cdots x_{j}$, which is in $I_{j-1}\subseteq
I_i$, by Lemma 3.8 in \cite{wiertel-JO}. The assertion follows.
\end{proof}

The following lemma provides a crucial step in the proof of
Theorem~\ref{main}. By ${\mathcal P}(K[C_n])$ we denote the prime
radical of $K[C_n]$.

\begin{lemma}\label{ann}
 Assume that $J$ is a finite dimensional ideal of $K[C_n]$. Then
$J=0$. In particular, the left annihilator $A=\{ \alpha \in
K[C_n]: \alpha K[M]=0\}$ of $K[M]$ in $K[C_n]$ is zero. Moreover,
$K[C_n]$ is a semiprime algebra.
\end{lemma}
\begin{proof}
Suppose that $J\neq 0$ is a finite dimensional ideal of $K[C_n]$.
First, we claim that a nonzero element $\alpha\in J$ can be chosen
so that for every $i=1,\ldots ,n$ we have $wx_i = w$ for all $w\in
\supp(\alpha) $ or $\alpha x_i=0$.

 Let $0 \neq\alpha\in J$ be such that $|\supp(\alpha)|$ is
minimal possible. Let $\supp (\alpha) =\{v_1,\ldots, v_k\}$. Since
$J$ is finite dimensional, the set $Z$ consisting of all such
$k$-tuples is finite.

Let $\mathcal R$ denote the Green's relation on the monoid $C_n$,
see \cite{clifford}. Consider the ${\mathcal R}$-order
$\leq_{\mathcal R}$ on $C_n$; in other words, we write
$w\leq_{\mathcal R} v$ if $wC_n\subseteq vC_n$. Then define a
relation $\preceq$ on $C_{n}^k$ by: $(u_1,\ldots ,u_{k})\preceq
(w_1,\ldots ,w_k)$ if $u_i \leq_{\mathcal R} w_i$ for every
$i=1,\ldots ,k$.

Now, by the choice of $\alpha$, for every $x\in C_n$ we have that
either $\alpha x=0$ or $\supp(\alpha x) =\{v_{1}x, \ldots ,v_kx\}$
and in the latter case $(v_1x,\ldots, v_kx)\preceq (v_1,\ldots,
v_k)$. Since the set $Z$ introduced above is finite, we may
further choose an element $\alpha$ for which the $k$-tuple
$(v_1,\ldots, v_k)$ is minimal possible with respect to $\preceq$.
Then $v_i{\mathcal R}v_ix$ for every $i$. Since the monoid $C_n$
is ${\mathcal J}$-trivial by \cite{denton}, Theorem~4.5.3, it
follows that for every $j$ we either have $wx_j=w$ for every $w\in
\supp (\alpha)$ or $\alpha x_j =0$, as claimed.

Next, assume that $\beta \in K[C_n]$ is a nonzero element such
that $wx_{1}=w$ holds in $C_n$ for every $w\in \supp (\beta)$.
Then $|w|_{1}>0$ for every such $w$. Write $w=w_0 x_1w_1$, for
some reduced words $w_0,w_1$ such that $|w_{1}|_{1}=0$. We claim
that then $|w_{1}|_{n}=0$. Indeed, if $w_1=ux_nv$ with
$|v|_{n}=0$, then $wx_1=w_0x_1ux_nvx_1$ and then the only possible
reduction that allows to decrease the length of this word (needed
in order to get $wx_1=w$ in $C_n$) is of the form
$x_1zx_1\rightarrow zx_1$, where $z$ is a prefix of $ux_nv$
containing $ux_n$. But then we do not get $wx_1=w$ in $C_n$
because $x_1$ appears after the last occurrence of $x_n$ in the
reduced form of $wx_1$), a contradiction. So $|w_{1}|_{n}=0$, as
claimed.

Assume first that $|w|_{n}>0$. Write $w= sx_ntx_1w_1$, for some
reduced words $s,t$ (so $w_0=sx_nt$) such that $|t|_{n}=0$. Then
also $|t|_{1}=0$ because $w$ is reduced.  Hence, either $wx_n=
sx_ntx_1w_1x_n$ is a reduced word with $|wx_n|_{n}\geqslant 2$ (if
$|tw_1|_{n-1}>0$) or $wx_n=w$ in $C_n$ and the reduced form of
$wx_{n}=w$ does not end with generator $x_n$ (if
$|tw_1|_{n-1}=0$).

Next, consider the case when $|w|_{n}=0$. It is clear that in this
case $wx_n$ is a reduced word, and $|wx_n|_{n}=1$. Together with
the previous paragraph of the proof this implies that $wx_n\neq
w'x_n$ in $C_n$ for all $w,w'\in \supp(\beta) $ with $w\neq w'$.

We have proved that the hypotheses on $\beta $ imply that $\beta
x_n\neq 0$.

Now, we apply this observation to the element $\alpha$. Because of
the choice of $\alpha$, we get that if $\alpha x_1= \alpha$ then
$\alpha x_n=\alpha$. Using the automorphism $\sigma$ (and noting
that $\sigma(\alpha)$, as an element of the finite dimensional
ideal $\sigma(J)$ of $K[C_n]$, inherits the hypotheses on
$\alpha$) we get that $\sigma (\alpha)x_1 =\sigma (\alpha)$, so
that $\sigma (\alpha)x_n =\sigma (\alpha)$, by the above argument
applied to $\sigma(\alpha)$ in place of $\alpha$. Thus, $\alpha
x_{n-1}=\alpha$, by applying $\alpha^{-1}$. Repeating this
argument several times, we then get $\alpha x_{j} =\alpha$ for
every $j$. A similar argument shows that if $\alpha x_{k}\neq
0$ for some $k$, then $\alpha x_j\neq 0$ for every $j$. However,
$\alpha =\alpha x_nx_1x_2\cdots x_{n-1}\in J\cap
K[\tilde{M}_{n-2}]$, a finite dimensional ideal of
$K[\tilde{M}_{n-2}]$, because $x_nx_1\cdots x_{n-1}\in
\tilde{M}_{n-2}\subseteq M$ and $\tilde{M}_{n-2}$ is an ideal of
$C_n$. Therefore, Lemma~\ref{matrix type} implies that $\alpha
=0$. This contradiction shows that we may assume that $\alpha
x_{j}=0$ for every $j$.

Let $w\in \supp(\alpha)$ be maximal with respect to the order
$\leq_{\mathcal R}$. If $x_j$ is the last letter of the (reduced
form of the) word $w$ then $w=wx_j=w'x_{j}$ in $C_n$, for some
$w'\in \supp(\alpha)$. This implies that $w\leq_{\mathcal R} w'$,
so by the choice of $w$ we get $w=w'$, a contradiction. Therefore
$J =0$.

By Theorem 5.9 in \cite{wiertel-JO}, ${\mathcal P}(K[C_n])\cap
K[\tilde{M}]=0$ because $\tilde{M}=\bigcup_{i=0}^{n-2}\tilde{M_i}$
and every $K[M_{i}]$ is prime. So, we know that $A\cap
K[\tilde{M}]=0$ and therefore $A$ and ${\mathcal P}(K[C_n])$ are
finite dimensional, because $C_n\setminus \tilde{M}$ is finite.
Hence, the assertion follows.
\end{proof}

We are now in a position to prove Theorem~\ref{main}.

\begin{proof}
In view of Lemma~\ref{ann}, from Theorem 5.9 in \cite{wiertel-JO}
we know that $K[C_n]$ is a Noetherian semiprime PI-algebra.

For any fixed $i=0,\ldots, n-2$, let $J_i$ be a maximal among all
ideals of $K[C_n]$ intersecting $K[x_nq_i]$ trivially. Then $J_i$
is a prime ideal, and $K[I_{i}]\subseteq J_i$. By Corollary 10.16
in \cite{krause}, $\GK (R)= \clK (R)$ (the Gelfand-Kirillov and
the classical Krull dimensions) for every finitely generated
Noetherian PI-algebra $R$. Since $\GK (K[C_n])=1$, it follows that
$J_i$ is a minimal prime ideal of $K[C_n]$. Clearly, the image
$J_{i}'$ of $J_i$ in $K[C_{n}]/K[I_i]$ is a prime ideal. $M_{i}$
is a right ideal in $C_{n}/I_i$ by Lemma~\ref{ideal}, and thus it
is a two-sided ideal because $C_{n}/I_i$ is endowed with a natural
involution which preserves $M_i$, by Corollary 3.12 and
Lemma 3.18 in \cite{wiertel-JO}. Since $K[M_i]$ is a prime
algebra, it follows that the classical quotient rings of $K[M_i]$
and $K[C_n]/J_i$ are equal. Moreover, as explained in the
introduction, the classical ring of quotients of $K[M_i]$ is
naturally isomorphic to $M_{n_i}(K(x))$, where $n_i =|A_i|$ for
$i=0,\ldots, n-2$. Therefore, $J=\bigcap_{i=0}^{n-2} J_i$ is a
semiprime ideal of $K[C_n]$ such that $J\cap K[M] =0$ (by the
definition of the ideals $J_i$). Since $C_n\setminus M$ is finite,
$J$ is finite dimensional, whence $J=0$ by Lemma~\ref{ann}. We
obtain that the quotient ring $Q$ of $K[C_n]$ satisfies $Q \cong
\prod_{i=0}^{n-2} M_{n_i}(K(x))$, $i=0,\ldots, n-2$. In view of
Proposition~\ref{rozmiar}, this completes the proof.
\end{proof}

Our second main result describes the radical of a Hecke--Kiselman
algebra $K[\HK_{\Theta}]$, as well as the algebra modulo the
radical, in the case of PI-algebras. So, assume that $\Theta $ is
a finite oriented graph such that $K[\HK_{\Theta}]$ is a
PI-algebra. This is equivalent to saying that $\Theta$ does not
contain two cyclic subgraphs (i.e. subgraphs which are cycles)
connected by an oriented path, \cite{mecel_okninski1}. Let $\rho$
be the congruence on $\HK_{\Theta}$ generated by all pairs
$(xy,yx)$ such that there is an arrow $x\rightarrow y$  that is
not contained in any cyclic subgraph of $\Theta$. (If there is no
such a pair then we assume that $\rho$ is the trivial congruence.)
Let $\Theta'$ be the subgraph of $\Theta$ obtained by deleting all
arrows $x\rightarrow y$ that are not contained in any cyclic
subgraph of $\Theta$. Then $\HK_{\Theta'}\cong \HK_{\Theta}/\rho$.
Then the connected components of $\Theta'$ are either singletons
or cyclic subgraphs.

Now, we are in a position to prove Theorem~\ref{second_main}.
\begin{proof} Suppose that a vertex $x\in V(\Theta)$ is a source
vertex. In other words, there is an arrow $x\rightarrow y$ for
some $y\in V(\Theta)$ but there are no arrows of the form
$z\rightarrow x$. For any $w\in \HK_{\Theta}$ consider the element
$\beta= (xy-yx)w (xy-yx)\in K[\HK_{\Theta}]$. Since $x$ is a
source vertex, we know that $xvx=xv$ in $\HK_{\Theta}$ for every
$v\in \HK_{\Theta}$. Hence $xwxy = xwy, xwyx=xwy$. Similarly, $
xywxy = xywy$ and $xywyx=xywy$. Therefore $\beta =0$. It follows
that $xy-xy\in {\mathcal P}(K[\HK_{\Theta}])$.

 If $x$ is a sink, that is there is an arrow $z\rightarrow x$ for
    some $z\in V(\Theta)$ but there are no arrows of the form
    $x\rightarrow y$ in the graph $\Theta$, a symmetric argument shows
    that $xz-zx \in {\mathcal P}(K[\HK_{\Theta}])$ for all $z$ such that
    $z\rightarrow x$ in $\Theta$.
Let $\rho_1$ be the congruence generated by all pairs $(xy, yx)$
such that $x$ or $y$ is either source or sink and there is an
arrow $x\rightarrow y$ that is not contained in any cyclic
subgraph of $\Theta$. Equivalently, we may consider the graph
$\Gamma_1$ obtained by erasing in $\Theta$ all such arrows
$x\rightarrow y$ and $z\rightarrow x$ as above. Then
$K[\HK_{\Gamma_1}]\cong K[\HK_{\Theta}]/I(\rho_1)$. We have shown
that $I(\rho_1)\s\mathcal{P}(K[\HK_{\Theta}])$. Repeating this
argument finitely many times we easily get that
$I(\rho)\subseteq{\mathcal P}(K[\HK_{\Theta}])$ (and our argument
shows that $I(\rho)$ is nilpotent, because $\Theta$ is finite).

 Since we know that ${\mathcal J}(K[\HK_{\Theta}])= {\mathcal
P}(K[\HK_{\Theta}])$, to prove the first assertion of the theorem
it is now enough to check that $K[\HK_{\Theta'}]$ is semiprime.
$\HK_{\Theta'}$ is the direct product of all $\HK_{\Theta_{i}}$,
where $\Theta_i$, $i=1,\ldots ,m$, are the connected components of
$\Theta'$ . From \cite{mecel_okninski1} we know that each
$\HK_{\Theta_{i}}$ is either a band with two elements (if
$\Theta_i$ has only one vertex) or it is isomorphic to $C_k$ for
some $k\geqslant 3$. In the former case $K[\HK_{\Theta_i}]\cong
K\oplus K$, in the latter $K[\HK_{\Theta_i}]$ is a semiprime
PI-algebra (by Theorem~\ref{main}) of Gelfand-Kirillov dimension
one \cite{mecel_okninski1}, and hence it is a finitely generated
module over its center \cite{stafford}. It follows easily that
$K[\HK_{\Theta}]$ is a finitely generated module over its center.

Let $Q_i$ be the classical  ring of quotients of
$K[\HK_{\Theta_i}]$. If $\Theta_i=C_{m_i}$ for some $m_i$ then we
know that  $Q_i$  is a central localization of the form described
in Theorem~\ref{main}. Clearly, $\HK_{\Theta'}$ is the direct
product $\prod_{i=1}^{m} \HK_{\Theta_i}$. Then in the localization
$Q=Q_1 \otimes \cdots \otimes Q_m$ of $K[\HK_{\Theta'}]\cong
\bigotimes_{i=1}^{m} K[\HK_{\Theta_i}]$ each of the factors is
isomorphic to $K\oplus K$ or to $\prod_{j=0}^{m_i-2}
M_{r_j}(K(x))$, where $r_j={m_i\choose j+1}$. Therefore, the
tensor product is semiprime. Hence $K[\HK_{\Theta'}]$ is
semiprime, because $Q$ is its central localization. It is now
clear that $K[\HK_{\Theta'}]\cong K[\HK_{\Theta}]/{\mathcal
P}(K[\HK_{\Theta}])$. The result follows.
\end{proof}

 \noindent {\bf Acknowledgment.} This work was
supported by  grant 2016/23/B/ST1/01045 of the National Science
Centre (Poland).

\vspace{20pt}

\begin{tabular}{lll}
Jan Okni\'nski & \quad \quad \quad \quad \quad & Magdalena Wiertel \\
\texttt{okninski@mimuw.edu.pl} & & \texttt{M.Wiertel@mimuw.edu.pl} \\
 & & \\
Institute of Mathematics & & \\
University of Warsaw & & \\
Banacha 2 & & \\
02-097 Warsaw, Poland & &
\end{tabular}

\end{document}